\documentclass{article}
\PassOptionsToPackage{numbers, compress}{natbib}
     \usepackage[preprint]{neurips_2020}

\usepackage[utf8]{inputenc} 
\usepackage[T1]{fontenc}    
\usepackage{hyperref}       
\usepackage{url}            
\usepackage{booktabs}       
\usepackage{amsfonts}       
\usepackage{nicefrac}       
\usepackage{microtype}      
\usepackage{microtype}
\usepackage{graphicx}
\usepackage{subfigure}
\usepackage{booktabs} 

\usepackage{gensymb}

\usepackage{amsmath,amsfonts,amssymb,bbm,amsthm,multirow,wrapfig}
\usepackage{hyperref}

\def\sD{\mathcal{D}}
\def\sH{\mathcal{H}}
\def\sV{\mathcal{V}}
\def\sX{\mathcal{X}}
\def\sT{\mathcal{T}}

\def\sR{\mathcal{R}}
\def\sP{\mathcal{P}}
\def\sQ{\mathcal{Q}}
\def\sF{\mathcal{F}}
\def\sB{\mathcal{B}}
\def\sW{\mathcal{W}}
\def\sL{\mathcal{L}}

\def\tsP{\widetilde{\sP}}
\def\tsH{\widetilde{\sH}}
\def\tsX{\widetilde{\sX}}

\def\bv{\mathbf{v}}
\def\be{\mathbf{e}}

\def\tsT{\widetilde{\mathcal{T}}}

\def\tf{\widetilde{f}}
\def\tp{\widetilde{p}}
\def\tq{\widetilde{q}}
\def\tw{\widetilde{w}}

\def\tP{\widetilde{P}}

\def\th{\widetilde{h}}
\def\tX{\widetilde{X}}

\def\tW{\widetilde{W}}
\def\kmax{k_{\text{max}}}
\def\bmax{b_{\text{max}}}
\def\1{\mathbbm{1}}
\newcommand{\rn}{\mathbb{R}}

\def\cip{\,{\buildrel p \over \rightarrow}\,}

\def\ot{\otimes}

\def\E{\mathbb{E}}

\def\conv{\operatorname{Conv}}
\def\mix{\operatorname{-mix}}

\def\simiid{\overset{iid}{\sim}}

\newtheorem{lem}{Lemma}[section]
\newtheorem{thm}{Theorem}[section]
\newtheorem{cor}{Corollary}[section]

\title{Improving Nonparametric Density Estimation with Tensor Decompositions}

\author{%
  Robert A.~Vandermeulen \\
  Machine Learning Group\\
  Technische Universit\"at Berlin\\
  Berlin, Germany \\
  \texttt{vandermeulen@tu-berlin.de} \\
}
\begin{document}
\maketitle
\begin{abstract}

  While nonparametric density estimators often perform well on low dimensional data, their performance can suffer when applied to higher dimensional data, owing presumably to the curse of dimensionality. One technique for avoiding this is to assume no dependence between features and that the data are sampled from a separable density. This allows one to estimate each marginal distribution independently thereby avoiding the slow rates associated with estimating the full joint density. This is a strategy employed in naive Bayes models and is analogous to estimating a rank-one tensor. In this paper we investigate whether these improvements can be extended to other simplified dependence assumptions which we model via nonnegative tensor decompositions. In our central theoretical results we prove that restricting estimation to low-rank nonnegative PARAFAC or Tucker decompositions removes the dimensionality exponent on bin width rates for multidimensional histograms. These results are validated experimentally with high statistical significance via direct application of an existing nonnegative tensor factorization to histogram estimators.
\end{abstract}
\section{Introduction}\label{sec:introduction}
\emph{Nonparametric density estimation} is the task of estimating a density $f$ from data without assuming that $f$ belongs to some parametric class of densities, e.g. the space multivariate Gaussian distributions. Some common nonparametric density estimators include the histogram estimator and the kernel density estimator (KDE). While nonparametric density estimation has been shown to be effective for many tasks, it has been observed empirically that estimator performance typically declines as data dimensionality increases, a manifestation of the curse of dimensionality. For the histogram and KDE this phenomena also has a precise mathematical analog. For these estimators universal consistency is achieved iff $n\to \infty$ and $h\to 0$,  with $nh^d \to \infty$, where $n$ is number of samples and $h$ is the bin width for the histogram and bandwidth parameter for the KDE \cite{gyorfi85}. 

One of the most common approaches to alleviating the curse of dimensionality is \emph{dimensionality reduction}. A dimensionality reduction technique typically attempts to transform a high dimensional representation into a lower dimensional one that removes dependences within data. Feature selection techniques, for example, usually explicitly remove highly dependent features \cite{song12,chen17}. In PCA one finds the best $d$-dimensional affine subspace for approximating data. If the PCA model fits well, it implies that, given $d$ features of a $D$-dimensional sample, there exists a good linear prediction of the remaining $D-d$ features. The \emph{manifold hypothesis} \cite{cayton05}, which is often touted as a general explanation for why high dimensional problems are learnable \cite{bengio13}, can also be viewed as a model of dependence. For example, if we assume that a random vector $[X,Y]^T$ lies on the one-dimensional sphere, $X^2+Y^2 = 1$, and $Y$ is known, then $X$ can only assume one of two values $X = \pm \sqrt{1-Y^2}$. More generally, the assumption that the data lies on a manifold implies local dependence. This is because any sufficiently small region of the support manifold can be well approximated by an affine subspace, and thus there exists a linear dependence between the features like in PCA. Interestingly, assuming that features are \emph{not} dependent yields another approach to overcoming the curse of dimensionality. 

In a naive Bayes model one assumes no dependence between features, i.e. our target density is separable,
\begin{align}\label{eqn:nb}
  f\left(x_1,\ldots,x_d\right) = f_1\left(x_1\right)f_2\left(x_2\right) \cdots f_d\left(x_d\right).
\end{align}

In order to estimate $f$ one can now simply estimate the marginal distributions and multiply them. Because the dimensionality of each marginal is one, one can use a histogram or KDE and achieve $nh\to \infty$ rates on bin width/bandwidth while preserving consistent estimation, thereby circumventing the curse of dimensionality. The separability assumption is rarely satisfied in practice so naive Bayes models are typically not used for density estimation directly, but may be used for some other task such as classification via a likelihood ratio test.

In this paper we consider relaxations of the naive Bayes model based on the assumption that a density is a mixture of separable densities. Our models are inspired by nonnegative tensor factorizations so we term them generally \emph{nonparametric nonnegative tensor factorization} (NNTF) \emph{models}. The first model is related to nonnegative PARAFAC \cite{shashua05} and is commonly known as a \emph{multi-view model} in statistics or machine learning literature:
\begin{equation}
  f\left(x_1,\ldots,x_d\right) = \sum_{i=1}^k w_i f_{i,1}\left(x_1\right)f_{i,2}\left(x_2\right) \cdots f_{i,d}\left(x_d\right).
\end{equation}
This is equivalent to the assumption that the features are independent conditioned on an unobserved discrete random variable taking on $k$ values.
Our second model is based on the nonnegative Tucker decomposition \cite{kim07}. In this model it is assumed that there are $d$ collections of $k$ one-dimensional densities, $\sF_1,\ldots, \sF_d$ with $\sF_i = \left\{f_{i,1},\ldots, f_{i,k} \right\}$, and some probability measure which which randomly selects one density from each collection. This measure a can be represented by a tensor $W \in\rn^{k^{\times d}}$ where the probability of selecting $f_{1,i_1},\ldots, f_{d,i_d}$ is $W_{i_1,\ldots,i_d}$. To sample from this model we first randomly select the marginal distributions $f_{1,i_1},\ldots,f_{d,i_d}$ according to $W$ and then independently sample each feature according the randomly selected marginal distribution $X_j \sim f_{j,i_j}$. The density of this model is
\begin{equation}
  f\left(x_1,\ldots,x_d\right)
  \quad= \sum_{i_1 =1}^k \cdots \sum_{i_d =1}^k W_{i_1,\ldots,i_d}f_{1,i_1}(x_1)\cdots f_{d,i_d}(x_d).\label{eqn:tucker}
\end{equation}
We are unaware of previous literature investigating this model so we will simply term it the \emph{Tucker model}.

 In Section \ref{sec:theory} we prove that there exists a trade-off between the rate on bin width $h$ and number of components $k$: to control estimation error\footnote{For an estimator $V$ restricted space of densities $\sP$, the \emph{estimation error} refers the difference between $\left\|V-p\right\|$ and $\min_{q\in \sP} \left\|q - p\right\|$, where $p$ is the target density.} we require $k/h$ to be asymptotically dominated by $n$ for the multi-view histogram and $k/h + k^d$ to be asymptotically dominated by $n$ for the Tucker histogram (both of these rates ignore logarithmic factors). Note that for the multi-view histogram this rate is not dependent on $d$. Allowing $h$ to shrink as aggressively as possible (which we pay for with a slow rate on $k$) we show that there exist \emph{universally consistent} histogram estimators which achieve $nh/\log\left(h^{-1}\right) \to \infty$ rate on bin width, thereby removing the dependence on dimension and approximately attaining rates possible for densities of the form in (\ref{eqn:nb}) while still controlling estimation error. We show that these are the approximately fastest possible rates via matching lower bound.
In Section \ref{sec:experiments} we show that we can use an existing low-rank nonnegative Tucker factorization algorithm to fit our model and demonstrate empirically that fitting histograms to a Tucker model outperforms the standard histogram estimator with very high statistical significance.

While this paper focuses on histogram estimation and presents a promising, readily implementable improvement to the standard histogram estimator in Section \ref{sec:experiments}, its primary purpose is to showcase the potential of utilizing concepts nonnegative tensor factorization to improve performance in nonparametric statistical methods.

\subsection{Previous Work}

Nonparametric density estimation has been extensively studied with the histogram estimator and KDE being by far the most well known. There do exist, however, alternative methods for density estimation, e.g. the \emph{forest density estimator} \cite{liu11} and \emph{$k$-nearest neighbor density estimator} \cite{mack79}. The $L_1,L_2$ and $L_\infty$ convergence of the histogram and KDE has been studied extensively \cite{gyorfi85,devroye01,tsybakov08,jiang17}. The KDE is generally accepted as being the superior density estimator with some mathematical justification \cite{silverman86}. Numerous modifications and extensions of the KDE have been proposed including utilizing variable bandwidth \cite{terrell92}, robust KDEs \cite{kim12,vandermeulen13,vandermeulen14}, and methods for enforcing support boundary constraints \cite{schuster85}. Finally we mention one recent paper \cite{kim18} that demonstrated that uniform convergence of a KDE to its population estimate suffered when the intrinsic dimension of the data was lower than the ambient dimension, a phenomenon seemingly at odds with the curse of dimensionality. 

For our review of NNTF models we also include a general review of tensor/matrix factorizations since both can be viewed being low-rank models. In particular, for the multi-view model we have
\begin{equation}\label{eqn:lrtens}
	\sum_{i=1}^k w_i f_{i,1}\left(x_1\right) \cdots f_{i,d}\left(x_d\right) \sim \sum_{i=1}^k \lambda_i \bv_{i,1}\otimes  \cdots \otimes \bv_{i,d}.
\end{equation}
We further note that for histogram estimation, once data has been assigned to bins, finding a good multi-view or Tucker histogram is analogous to estimating a probability tensor with a nonnegative factorization (we show this rigorously in Section \ref{sec:experiments}).

A great deal of work has gone into leveraging low-rank assumptions to improve matrix estimation, particularly in the field of \emph{compressed sensing} \cite{donoho06,recht10}. In compressed sensing one has access to a collection of random linear measurements of an unknown low-rank matrix to be estimated. Fitting a matrix to these measurements with a nuclear norm regularized optimization problem achieves estimation bounds better than those possible without the low-rank assumption. These techniques have been effectively applied to problems such as matrix completion, multivariate regression, and estimating autoregressive models \cite{negahban11,negahban12}. Unfortunately such techniques are not extensible to histogram estimation because, in the density estimation setting, data are not linearly sampled from the target model. Furthermore how to extend compressed sensing techniques to general tensors is not clear.

General matrix/tensor factorization, including nonnegative matrix/tensor factorizations, have been extensively studied despite their inherent difficulty due to non-convexity. The works \cite{donoho03,arora12} present potential theoretical grounds for avoiding the computational difficulties of nonnegative matrix factorization. Some algorithms for finding nonnegative tensor factorizations are mentioned in Section \ref{sec:experiments}. One notable approach to tensor factorization is to assume, in the tensor representation in (\ref{eqn:lrtens}), that $d \ge 3$ and the collections of vectors $\bv_{1,j},\ldots, \bv_{k,j}$ are linearly independent for all $j$. Under this assumption we are guaranteed that the factorization (\ref{eqn:lrtens}) is unique \cite{allman09}. In \cite{anandkumar14} the authors present a method for recovering this factorization efficiently and demonstrate its utility for a variety of tasks. This work was extended in \cite{song14} to recover a multi-view KDE satisfying an analogous linear independence assumption. This is the only work of its type of which we are aware. In this work the authors investigate the sample complexity of their estimator but do not demonstrate that their technique has potential for improving rates for nonparametric density estimation in general. Finally we note that nonparametric applications of the Tucker decomposition have been utilized in Bayesian statistics \cite{schein16}.  We are unaware of any literature describing the model we introduce in (\ref{eqn:tucker}).

\section{Theoretical Results}\label{sec:theory}
In this section we mathematically demonstrate that histogram estimators can achieve greater estimation accuracy by restricting to NNTF models. To simplify analysis we will only consider densities on $\left[0,1\right)^d$ and analyze number of bins per dimension $b$ which is the inverse of the bin width, i.e. $b = 1/h$. We prove that there exists a trade-off between rates on $b$ and $k$. Furthermore we show that the approximate fastest possible rate on $b$ while still uniformly controlling for estimator variance and remaining universally consistent is $n/\left(b\log b\right) \to \infty$.  Before proving these results we must introduce a fair amount of notation.
\subsection{Notation}
All norms in Section \ref{sec:theory} are the $\ell^1, L^1,$ or total variation norm for vectors/tensors, densities, and measures respectively. Note that these norms are analogous e.g. the $L^1$ norm of a probability density function is the same as the total variation norm on the probability measure associated with the density. Let $\sD_d$ be the set of all densities on $\left[0,1\right)^d$. By \emph{density} we mean probability measures that are absolutely continuous with respect to the $d$-dimensional Lebesgue measure. We define a \emph{probability vector} or \emph{probability tensor} to simply mean a vector or tensor whose entires are nonnegative and sum to one. Let $\Delta_b$ denote the set of probability vectors in  $\rn^b$and $\sT_{d,b}$ the set of probability tensors in $\rn^{ b^{\times d}}$. Let $\sT_{d,b}^k$ be the set of tensors which are a convex combination of $k$ separable probability tensors, i.e.

\begin{align*}
  \sT_{d,b}^k \triangleq \left\{ \sum_{i = 1}^k w_i \prod_{j=1}^d p_{i,j} \middle| w\in \Delta_k, p_{i,j} \in \Delta_b\right\}.
\end{align*}
In this paper, \emph{the product symbol $\prod$ will always mean the standard outer product}, e.g. set product or tensor product. For any natural number $b$ let $[b] =\left\{1,\ldots,b\right\}$. For a multi-index $A \in \left[b\right]^d$ we define $\be_{d,b,A}$ as the element of $\sT_{d,b}$ where the $(A_1,\ldots, A_d)$-th entry is one and is zero elsewhere.

The following is the set of probability tensors constructed via a nonnegative Tucker factorization
\begin{align*}
  \tsT_{d,b}^k \triangleq \left\{ \sum_{S \in \left[k\right]^{\times d}} W_S \prod_{j=1}^d p_{i,S_j} \middle| W\in \sT_{b,k}, p_{i,j} \in \Delta_b\right\}.
\end{align*}
We will now construct the space of histograms on $\left[0,1\right)^d$. We begin with one-dimensional histograms. We define $h_{1,b,i}$ with $i \in \left[b\right]$ to be the one dimensional histogram where all weight is allocated to the $i$th bin. Formally we define this as
\begin{align*}
  h_{1,b,i}\left(x\right) \triangleq b \1\left(\frac{i-1}{b}\le x <\frac{i}{b} \right).
\end{align*}
Note that this is a valid density due to the leading $b$ coefficient.
We use these to construct higher-dimensional histograms. For a multi-index $A\ \in \left[b\right]^d$, let 
\begin{align*}
  h_{d,b,A} \triangleq \prod_{i=1}^d h_{1,b,A_i},
\end{align*}
the histogram whose entire density is allocated to the bin indexed by $A$. Finally we define $\Lambda_{d,b,A}$ to be the support of $h_{d,b,A}$, i.e. the ``bins'' of a histogram estimator,
\begin{align*} 
  \Lambda_{d,b,A} \triangleq \prod_{i=1}^d \left[\frac{A_i-1}{b},\frac{A_i}{b}\right). %
\end{align*}

For a sequence of points in $\left[0,1\right)^d$, $\sX = \left(X_1,\ldots, X_n\right)$, the standard histogram estimator is 
\begin{align*}
  H_{d,b}\left(\sX\right) \triangleq \frac{1}{n}\sum_{i=1}^n\sum_{A \in [b]^d} h_{d,b,A}\1\left(X_i\in \Lambda_{d,b,A} \right).
\end{align*}

Let $\sH_{d,b} \triangleq \conv\left(\left\{h_{d,b,A} \middle| A \in \left[ b\right]^d \right\}\right)$, the set of all $d$-dimensional histograms with $b$ bins per dimension. Let $\sH_{d,b}^k$ be the set of histograms with at most $k$ separable components, i.e.
\begin{equation}\label{eqn:lrhist}
  \sH_{d,b}^k \triangleq \left\{\sum_{i=1}^k w_i\prod_{j=1}^d f_{i,j} \middle| w\in \Delta_k, f_{i,j}\in \sH_{1,b} \right\}.
\end{equation}
We will refer elements in this space \emph{multi-view histograms}. Analogously we define the space of \emph{Tucker histograms} to be
\begin{equation*}
  \tsH_{d,b}^k = \left\{\sum_{S \in \left[k\right]^{\times d}} W_S \prod_{i=1}^d f_{i,S_i}\middle| W \in \sT_{d,k}, f_{i,j} \in \sH_{1,b} \right\}.
\end{equation*}

We emphasize that the collections of densities $\sH_{d,b}^k$ and $\tsH_{d,b}^k$ the primary objects of interest in this paper as they represent NNTF histograms. The theoretical results we present are concerned with finding good density estimators restricted to these sets as $k$ and $b$ vary.

Note that there exists a $\ell^1 \to L^1$ linear isometry $U_{d,b}:\sT_{d,b} \to \sH_{d,b}$ with $U_{d,b}$ defined as
\begin{equation*}
  U_{d,b}(\be_{d,b,A}) = h_{d,b,A}.
\end{equation*}
The inverse function, $U^{-1}_{d,b}$, simply transforms a histogram to the tensor representing its bin weights and $U_{d,b}$ performs the reverse transformation. Note that $U_{d,b}$ is also a bijection between $\sT_{d,b}^k \to \sH_{d,b}^k$ and $\tsT_{d,b}^k \to \tsH_{d,b}^k$. Much of our analysis on histograms will be performed on the space of probability tensors with the analysis being translated to histograms via this operator.

For a set of vectors $\sV$ we define $k\mix\left(\sV\right)\triangleq \left\{\sum_{i=1}^k w_i v_i \middle| w\in \Delta_k, v_i \in \sV \right\}$, i.e. the set of convex combinations of collections of $k$ vectors from $\sV$. We define $N\left(\sV, \varepsilon \right)$ to be the minimum cardinality for a subset of of $\sV$ which $\varepsilon$-covers $\sV$ (with closed balls) with respect to the $\left\|\cdot \right\|$ metric. It will be clear from context whether $\left\|\cdot \right\|$ represents the $\ell^1$, $L^1$ or total variation norm.
\subsection{Preliminary Results}
For brevity the main text only contains a full proof of Lemma \ref{lem:finiteest}. The remaining full proofs can be found in the appendix. We include general descriptions of the proof techniques we use for multi-view histogram results. These are similar to the techniques we use for Tucker histograms. Our general proof technique is to find good covers of spaces of densities, i.e. $\sH_{d,b}^k$ and $\tsH_{d,b}^k$, and then apply an existing algorithm for selecting a good estimator from finite collections of densities given data. We begin by establishing a covering number bound on the space of probability vectors via an adaptation of a standard result presented in \cite{devroye01}.
\begin{lem} \label{lem:histcover}
  For all $0 < \varepsilon \le 1$ we have that 
  $N\left(\Delta_{b},\varepsilon\right)
  \le 
  \left(\frac{2b}{\varepsilon} \right)^b$.
\end{lem}

We can extend this to a covering number for the space of separable probability tensors.

\begin{lem}\label{lem:simplecover}
  \sloppy For all $0 < \varepsilon \le 1 $ we have that $N\left(\sT_{d,b}^1, \varepsilon \right) \le \left(\frac{2bd}{\varepsilon} \right)^{bd}$.
\end{lem}

\begin{proof}[Proof sketch]
  \sloppy Combine Lemma \ref{lem:histcover} with the following standard bound for product measures (see Lemma 3.3.7 in \cite{reiss89}):
  $\left\|\prod_{i=1}^d q_i -\prod_{j=1}^d \tq_j\right\| \le \sum_{i=1}^d \left\|q_i- \tq_i \right\|.$
\end{proof}
Now we establish the following lemma for covering numbers of mixtures of densities.
\begin{lem}\label{lem:mixcover}
  Let $\sP$ be a set of probability measures, then
  \begin{align*}
    N\left(k\operatorname{-mix}\left(\sP\right), \varepsilon + \delta \right) \le N\left(\sP,\varepsilon\right)^k N\left(\Delta_k,\delta\right).
  \end{align*}
\end{lem}
\begin{proof}[Proof sketch]
  Use Lemma \ref{lem:histcover} to construct different weightings of $k$ elements from an $\varepsilon$-cover of $\sP$.
\end{proof}
By combining Lemma \ref{lem:mixcover} with Lemma \ref{lem:simplecover} we arrive at covering numbers for the space of multi-view probability tensors.

\begin{lem}\label{lem:lrcover}
  For all $0 < \varepsilon \le 1$ the following holds $N\left(\sT_{d,b}^k, \varepsilon \right) \le \left(\frac{4bd}{\varepsilon} \right)^{bdk} \left( \frac{4k}{\varepsilon}\right)^k$.
\end{lem}
Through application of the $U_{d,b}$ operator we now have a characterization of the complexity of the space $\sH_{d,b}^k$.

\begin{cor}\label{cor:lrcover}
  For all $0 < \varepsilon \le 1$ following holds $N\left(\sH_{d,b}^k,\varepsilon\right)\le \left(\frac{4bd}{\varepsilon} \right)^{bdk} \left( \frac{4k}{\varepsilon}\right)^k.$
\end{cor}
The following are analogous results for Tucker histograms.
\begin{lem}\label{lem:tuckcover}
  For all $0 < \varepsilon \le 1$ the following holds $N\left(\tsT_{d,b}^k, \varepsilon \right) \le \left(\frac{4bd}{\varepsilon} \right)^{bdk} \left( \frac{4k^d}{\varepsilon}\right)^{k^d}$.
\end{lem}

\begin{cor}\label{cor:tuckcover}
  For all $0 < \varepsilon \le 1$ following holds $N\left(\tsH_{d,b}^k,\varepsilon\right)\le \left(\frac{4bd}{\varepsilon} \right)^{bdk} \left( \frac{4k^d}{\varepsilon}\right)^{k^d}.$
\end{cor}

The following lemma from \cite{ashtiani18} provides us with a way to choose good estimators from finite collections of densities. It can be proven by applying a Chernoff bound to \cite{devroye01}, Theorem 6.3.
\begin{lem}[\cite{ashtiani18}] \label{lem:densalg}
  There exists a deterministic algorithm that, given a collection of distributions $p_1,\ldots,p_M$, a parameter $\varepsilon >0$ and at least $\frac{\log \left(3M^2/\delta \right)}{2\varepsilon^2}$ iid samples from an unknown distribution $p$, outputs an index $j\in \left[M\right]$ such that
  \begin{align*} 
    \left\|p_j - p \right\| \le 3 \min_{i \in \left[M\right]} \left\|p_i - p\right\| + 4 \varepsilon
  \end{align*}
  with probability at least $1 - \frac{\delta}{3}$.
\end{lem}

We present the following asymptotic version of the previous lemma and include its full proof. We highlight the use of finding sufficiently slow rates on parameters in order to establish asymptotic results, a technique which we will use in later proofs.

\begin{lem}\label{lem:finiteest}
  Let $\left( \sP_n\right)_{n \in \mathbb{N}}$ be a sequence of finite collections of densities in $\sD_d$ where $\left| \sP_n\right| \to \infty$ with $n/\log\left(\left|\sP_n\right|\right)\to \infty$. Then there exists a sequence of estimators $V_n \in \sP_n$ such that, for all $\gamma >0$,
  \begin{align*}
    \sup_{p \in \sD_d }P\left(\left\|V_n- p\right\| 
    > 3 \min_{q\in \sP_n}\left\|p -q \right\| + \gamma    \right) \to 0,
  \end{align*}
  where $V_n$ is a function of $X_1,\ldots,X_n \simiid p$.
\end{lem}
\begin{proof}
  Let $M = M(n)= \left| \sP_n\right|$. Since $n/\log\left(M\right) \to \infty$ we have that for all $c>0$ there exists a $N_c$ such that, for all $n\ge N_c$ we have $n/\log\left(M\right) \ge c $ or equivalently $n\ge c\log\left(M\right)$. Because of this there exists sequence of positive values $C = C(n)$ such that $C \to \infty$ and $n \ge C\log\left(M\right)$.

  We will be making use of the algorithm in Lemma \ref{lem:densalg} as well as its notation. If we can show that there exist sequences of positive values $\varepsilon(n)\to 0, \delta(n)\to 0$ such that, for sufficiently large $n$, the following holds
  \begin{align*}
    \frac{\log \left(3M^2/\delta \right)}{2\varepsilon^2}\le n,
  \end{align*}
  then can simply set $V_n$ equal to be the estimator from Lemma \ref{lem:densalg} for sufficiently large $n$ and, because the lemma holds independent of choice of $p$, the theorem statement follows.

  Let $\varepsilon = \left(2/C \right)^{1/4}$ and $\delta = 3/\left(\exp\left(2 \sqrt{\frac{C}{2}} \right) \right)$. Note that these are both positive sequences which converge to zero. Now we have
  \begin{align}
    \frac{\log\left(3M^2/\delta\right)}{2 \varepsilon^2}\notag 
    &= \frac{\log\left(M^2\right) + \log\left(3/\delta \right)}{2 \varepsilon^2}\\
    & = \frac{2\log\left(M\right) + \log\left(3/\delta \right)}{2 \varepsilon^2}\notag = \frac{\log\left(M\right) + \frac{1}{2}\log\left(3/\delta \right)}{\varepsilon^2}\notag\\
    & = \varepsilon^{-2}\left(\log\left(M\right) + \frac{1}{2}\log\left(3/\delta \right)\right)\notag\\
    &= \left(\left(2/C \right)^{1/4}\right)^{-2}\left(\log\left(M\right) + \frac{1}{2}\log\left(\exp\left(2 \sqrt{\frac{C}{2}} \right) \right) \right)\notag\\
    & = \sqrt{\frac{C}{2}} \left(\log(M)+ \sqrt{\frac{C}{2}} \right) = \sqrt{\frac{C}{2}} \log(M)+ \frac{C}{2}\label{eqn:finiteest}.
  \end{align}
  For sufficiently large $C$ and $M$ we have that the RHS of (\ref{eqn:finiteest}) is less than or equal to
  \begin{align*}
    \frac{C}{2} \log(M)+ \frac{C}{2}
    &\le \frac{C}{2} \log(M)+ \frac{C}{2}\log(M)	\\
    &= C\log(M) \le n.
  \end{align*}
  which completes our proof.
\end{proof}
\subsection{Main Theoretical Results}
We can now state the central results of this paper. The following theorem states that one can control the estimation error of multi-view histograms with $k$ components and $b$ bins per dimension so long as $n \sim bk$ (omitting logarithmic factors). Recall that the standard histogram estimator requires $n\sim b^d$, so we have removed the exponential dependence of bin rate on dimensionality. Here and elsewhere the $\sim$ symbol is not a precise mathematical statement but rather describes that the two values should be of the same order. In the following $b$ and $k$ are functions of $n$; the space of histograms which we are fitting changes as we acquire more data.
\begin{thm}\label{thm:genrate}
  For any pairs of sequences $b \to \infty$ and $k \to \infty$ satisfying $n/(bk\log(b) + k\log(k)) \to \infty$, there exists an estimator $V_n\in \sH_{d,b}^k$ such that, for all $\varepsilon >0$
  \begin{align*}
    &\sup_{p \in \sD_d }P\left(\left\|V_n - p\right\| >
    3 \min_{q\in \sH_{d,b}^k}\left\|p -q \right\| + \varepsilon  \right) \to 0,
  \end{align*}
  where $V_n$ is a function of  $X_1,\ldots, X_n\simiid p$.
\end{thm}
\begin{proof}[Proof sketch]
  Apply Lemma \ref{lem:finiteest} to the cover in Corollary \ref{cor:lrcover} and choose appropriately slow rates for terms not involving $b$ or $k$.
\end{proof}
The sample complexity for the multi-view histogram is perhaps more accurately approximated as being on the order of $dbk$ however the $d$ disappears in the asymptotic analysis.

The following theorem states that we can control the error of Tucker histogram estimates so long as $n\sim bk+k^d$ (omitting logarithmic factors).
\begin{thm}\label{thm:tuckrate}
  For any pairs of sequences $b \to \infty$ and $k \to \infty$ satisfying $n/(bk\log(b) + k^d\log\left(k^d\right)) \to \infty$, there exists an estimator $V_n\in \tsH_{d,b}^k$ such that, for all $\varepsilon >0$
  \begin{align*}
    &\sup_{p \in \sD_d }P\left(\left\|V_n - p\right\| >
    3 \min_{q\in \tsH_{d,b}^k}\left\|p -q \right\| + \varepsilon  \right) \to 0,
  \end{align*}
  where $V_n$ is a function of $X_1,\ldots, X_n\simiid p$.
\end{thm}
Allowing $b$ to grow as aggressively as possible we achieve consistent estimation, using either the multi-view or Tucker histograms, so long as $n \sim b\log b $ regardless of dimensionality.
\begin{cor}\label{cor:fastbin}
  \sloppy For all $d,b,k$ fix  $\sR_{d,b}^k$ to be either $\sH_{d,b}^k$ or $\tsH_{d,b}^k$. For any sequence $b \to \infty$ with $n/\left(b\log b\right) \to \infty$, there exists a sequence $k \to \infty$ and estimator $V_n\in \sR_{d,b}^k$ such that, for all $\varepsilon >0$
  \begin{align*}
    \sup_{p \in \sD_d }P\left(\left\|V_n - p\right\| 
   > 3 \min_{q\in \sR_{d,b}^k}\left\|p -q \right\| + \varepsilon  \right) \to 0,
  \end{align*}
  where $V_n$ is a function of  $X_1,\ldots, X_n\simiid p$.
\end{cor}

Replacing $b:= 1/h$ allows us to arrive at the rates mentioned in Section \ref{sec:introduction}.
The following result shows that the bias of the estimators in Theorems \ref{thm:genrate} and $\ref{thm:tuckrate}$ go to zero for all densities. Thus these estimators are universally consistent even when the NNTF model assumption is not satisfied.
\begin{lem}\label{lem:lrbias}
  Let $p \in \sD_d$. If $k\to \infty$ and $b \to \infty$ then $\min_{q\in \sH_{d,b}^k} \left\|p - q\right\| \to 0$.
\end{lem}
A straightforward consequence of this is that the Tucker histogram bias also goes to zero.
\begin{lem}\label{lem:tuckbias}
  Let $p \in \sD_d$. If $k\to \infty$ and $b \to \infty$ then $\min_{q\in \tsH_{d,b}^k} \left\|p - q\right\| \to 0$.
\end{lem}

Finally we have that rate on $bk$ in Theorem \ref{thm:genrate} cannot be made significantly faster.
\begin{thm}\label{thm:lower}
	Let $d \ge 2$, $b\to \infty$, and $k\to \infty$ with $b \ge k$ and $n/\left(bk\right) \to 0$. There exists no estimator $V_n\in \sH_{d,b}^k$ such that, for all $\varepsilon >0$, the following limit holds
	\begin{align*}
		\sup_{p \in \sD_d }P\left(\left\|V_n - p\right\| > 3 \min_{q\in \sH_{d,b}^k}\left\|p -q \right\| + \varepsilon  \right) \to 0
	\end{align*}
  where $V_n$ is a function of  $X_1,\ldots, X_n\simiid p$.
\end{thm}

\begin{proof}[Proof sketch]
  We use $V_n$ to construct an estimator over $\Delta_b$ and show that such an estimator is impossible using a result in \cite{han15}.
\end{proof}

Likewise the rate on $bk+k^d$ can also not be significantly improved in Theorem \ref{thm:tuckrate}.

\begin{thm}\label{thm:tucklower}
	Let $d \ge 2$, $b\to \infty$, and $k\to \infty$ with $b \ge k$ and $n/\left(bk+k^d\right) \to 0$. There exists no estimator $V_n\in \tsH_{d,b}^k$ such that, for all $\varepsilon >0$, the following limit holds
	\begin{align*}
		\sup_{p \in \sD_d }P\left(\left\|V_n - p\right\| > 3 \min_{q\in \tsH_{d,b}^k}\left\|p -q \right\| + \varepsilon  \right) \to 0
	\end{align*}
  where $V_n$ is a function of  $X_1,\ldots, X_n\simiid p$.
\end{thm}
\subsection{Discussion} \label{ssec:discussion}
  Naturally real world data likely never exactly satisfies the NNTF model assumption. Our results are meant to highlight a trade-off between model assumptions of smoothness (low $b$) and simple dependence between features (low $k$). Here we will explore this trade-off for multi-view histogram. Letting $k = b^d$ gives $\sH_{d,b}^k = \sH_{d,b}$ since we can allocate one component to each bin. Using the estimator from Theorem \ref{thm:genrate} gives a sample complexity of approximately $b^{d+1}$. Thus setting $k=b^d$ in the multi-view histogram gives us something which behaves similarly to the standard histogram estimator with a similar sample complexity. On the other hand setting $k=1$ gives a naive Bayes model with a sample complexity of approximately $b$. The Tucker histogram can be similarly analyzed with $k=b$ corresponding to the standard histogram. Thus we have a span of $k$ yielding different estimators with maximal $k$ corresponding to the standard histogram and minimal $k$ corresponding to a naive Bayes assumption. We observe in Section \ref{sec:experiments} that this trade-off is useful in practice: we virtually never want $k$ to be maximized.

\section{Experiments}\label{sec:experiments}
While Theorems \ref{thm:genrate} and \ref{thm:tuckrate} guarantee the existence of estimators which can effectively estimate NNTF models, these estimators are unfortunately computationally intractable. Fortunately there exist estimators which can be adapted to our problem setting, though they lack the theoretical guarantees of the algorithm described in Theorems \ref{thm:genrate} and \ref{thm:tuckrate}. Specifically we will utilize an existing algorithm for nonnegative tensor decompositions. Due to the difficulties of estimating $L^1$ distances between densities we will instead focus on estimates minimizing the $L^2$ norm. In this section inner products and norms will be $L^2$ for functions and $\ell^2$ for tensors i.e. standard euclidean norm or inner product applied to flattened tensors. We will again restrict our analysis to densities supported on $\left[0,1\right)^d$.

Consider the problem of finding some density estimator $\hat{p}$ with minimal $L_2$ distance to an unknown density $p$. This is equivalent to minimizing the squared $L^2$ loss:
\begin{align}
  &\int_{\left[0,1\right]^d} \left(p(x) - \hat{p}\left(x\right) \right)^2dx\\
  &\quad= \int_{\left[0,1\right]^d}\hat{p}\left(x\right)^2 dx - 2\int_{\left[0,1\right]^d} p(y)\hat{p}(y) dy
  + \int_{\left[0,1\right]^d} p(z)^2 dz.\label{eqn:pnorm}
\end{align}
Because the right term in (\ref{eqn:pnorm}) does not depend on $\hat{p}$ it can be ignored when finding optimal $\hat{p}$. The left term in (\ref{eqn:pnorm}) is known. The middle term in can be estimated with the following approximation
\begin{equation*}
  \int_{\left[0,1\right]^d} p(x)\hat{p}(x) dx = \E_{X\sim p}\left[ \hat{p}(X) \right] \approx \frac{1}{n}\sum_{i=1}^n \hat{p}\left(X_i\right)
\end{equation*}
\sloppy where $\sX = X_1,\ldots, X_n \simiid p$. We can use this to find a good estimate for $p$ in $\sR^k_{d,b}$ which represents $\sH^k_{d,b}$ or $\tsH^k_{d,b}$:

\begin{align}
  &\arg \min_{\hat{H} \in \sR_{d,b}^k} \int_{\left[0,1\right]^d} \left(\hat{H}\left(x\right) - \hat{p}\left(x\right) \right)^2dx\notag\\
  &= \arg \min_{\hat{H} \in \sR_{d,b}^k} \left<\hat{H},\hat{H}\right> - 2\int_{\left[0,1\right]^d} \hat{H}(x)p(x) dx\\
  &\approx \arg \min_{\hat{H} \in \sR_{d,b}^k} \left<\hat{H},\hat{H}\right> - 2\frac{1}{n}\sum_{i=1}^n \hat{H}(X_i). \label{eqn:l2obj}
\end{align}

\sloppy Recall that the standard histogram estimator is $H\left(\sX\right) = \frac{1}{n}\sum_{i=1}^n\sum_{A \in \left[b\right]^d}  h_{d,b,A} \1\left(X_i \in \Lambda_{d,b,A}\right)$ and let $\hat{H} = \sum_{A\in \left[b\right]^d}\hat{w}_A h_{d,b,A}$.
We have the following
\begin{align*}
  \left<\hat{H},H\right>
  & =  \left<\sum_{A \in \left[b\right]^d} \hat{w}_A h_{d,b,A},
  \frac{1}{n}\sum_{i=1}^n \sum_{B\in \left[b\right]^d} h_{d,b,B} \1\left(X_i \in \Lambda_{d,b,B}\right) \right>\\
  & =  \frac{1}{n}\sum_{i=1}^n \sum_{A\in \left[b\right]^d}  \hat{w}_A \1\left(X_i \in \Lambda_B\right)b^d =  \frac{1}{n}\sum_{i=1}^n \hat{H}(X_i).
\end{align*}
As a consequence (\ref{eqn:l2obj}) is equal to
\begin{align*}
  \arg \min_{\hat{H} \in \sR_{d,b}^k}\left<\hat{H},\hat{H}\right> -2\left<H,\hat{H}\right>
  &=\arg \min_{\hat{H} \in \sR_{d,b}^k}\left<\hat{H},\hat{H}\right> -2\left<H,\hat{H}\right> + \left<H,H \right>\\
  &= \arg \min_{\hat{H} \in \sR_{d,b}^k}\left\|H-\hat{H}\right\|_2^2. 
\end{align*}
Using the $U_{d,b}$ operator we can reformulate this into a tensor factorization problem
\begin{equation*}
  \min_{\hat{T} \in \sQ_{d,b}^k}\left\|H-U_{d,b}(\hat{T})\right\|_2^2
  = \min_{\hat{T} \in \sQ_{d,b}^k}b^d\left\|U_{d,b}^{-1}\left(H\right)-\hat{T}\right\|_2^2.
\end{equation*}
Where $\sQ_{d,b}^k$ could be either $\sT_{d,b}^k$ or $\tsT_{d,b}^k$. Because of this equivalence, to find estimates in $\sH_{d,b}^k$ or $\tsH_{d,b}^k$ we can simply use nonnegative tensor decomposition algorithms, which minimize $\ell^2$ loss, to find NNTF tensors which approximate $H$.
\begin{table*}[h]
  \scriptsize
  \label{tab:results}
  \centering
  \caption{Experimental Results}
  \begin{tabular}{|c| *{8}{c|}}
    \hline
    Dataset &$d$ Red.&Dim.& Hist. Perf. & Tucker Perf. & Hist. Bins & Tucker Bins &Tucker $r$ & p-val.\\
    \hline \hline
    \multirow{8}{*}{MNIST}
    &
    \multirow{4}{*}{PCA}
    &2&-1.455$\pm$0.089&-1.502$\pm$0.102&6.531$\pm$1.499&8.375$\pm$1.780&4.968$\pm$1.976&5e-4\\
    &&3&-2.040$\pm$0.196&-2.268$\pm$0.195&4.781$\pm$0.738&6.718$\pm$1.565&5.781$\pm$1.340&2e-4\\
    &&4&-3.532$\pm$0.996&-4.014$\pm$0.655&4.031$\pm$0.585&5.343$\pm$1.018&4.375$\pm$0.695&2e-3\\
    &&5&-4.673$\pm$1.026&-6.157$\pm$2.924&3.468$\pm$0.499&4.343$\pm$0.592&3.281$\pm$0.514&4e-5\\\cline{2-9}

    &\multirow{4}{*}{Rand.}
    &2&-2.034$\pm$0.100&-2.099$\pm$0.102&6.062$\pm$1.197&7.562$\pm$1.657&2.062$\pm$1.784&3e-5\\
    &&3&-3.086$\pm$0.207&-3.331$\pm$0.387&4.812$\pm$0.526&6.843$\pm$1.227&2.687$\pm$1.959&1e-4\\
    &&4&-4.307$\pm$0.290&-5.731$\pm$0.435&3.500$\pm$0.559&5.656$\pm$0.642&2.593$\pm$1.497&8e-7\\
    &&5&-6.327$\pm$0.522&-9.539$\pm$1.053&3.250$\pm$0.433&4.718$\pm$0.571&2.562$\pm$1.087&8e-7\\
    \hline
    
    \multirow{8}{*}{Diabetes}
    &
    \multirow{4}{*}{PCA}
    &2&-2.079$\pm$0.122&-2.212$\pm$0.132&5.718$\pm$1.304&7.468$\pm$1.478&1.062$\pm$0.242&8e-6\\
    &&3&-3.010$\pm$0.364&-3.606$\pm$0.420&3.593$\pm$0.860&7.062$\pm$1.058&1.843$\pm$1.543&2e-6\\
    &&4&-4.002$\pm$0.415&-4.423$\pm$0.701&3.000$\pm$0.000&5.906$\pm$0.804&2.343$\pm$1.107&2e-3\\
    &&5&-6.139$\pm$0.661&-6.043$\pm$1.192&3.000$\pm$0.000&3.750$\pm$0.968&1.843$\pm$0.617&0.91\\\cline{2-9}

    &\multirow{4}{*}{Rand.}
    &2&-3.074$\pm$0.224&-3.277$\pm$0.287&6.843$\pm$1.227&9.250$\pm$1.936&1.093$\pm$0.384&7e-5\\
    &&3&-4.726$\pm$0.483&-5.353$\pm$0.751&4.968$\pm$0.769&8.406$\pm$1.343&1.625$\pm$1.672&2e-5\\
    &&4&-6.017$\pm$0.873&-7.732$\pm$1.497&4.062$\pm$0.704&6.718$\pm$1.328&2.093$\pm$1.155&1e-5\\
    &&5&-8.986$\pm$1.292&-12.61$\pm$2.477&3.062$\pm$0.242&5.093$\pm$0.521&2.531$\pm$0.865&2e-6\\
    \hline
  \end{tabular}
\end{table*}
For our experiments we used the Tensorly library \cite{tensorly} to perform the nonnegative Tucker decomposition \cite{kim07} with Tucker rank $[k,k,\ldots,k]$ which was then projected to the simplex of probability tensors using \cite{duchi08}. We also performed experiments with nonnegative PARAFAC decompositions using \cite{shashua05,tensorly}. These decompositions performed poorly. This is potentially because the PARAFAC optimization is more difficult or the additional flexibility of the Tucker decomposition was more appropriate for the experimental datasets.
\subsection{Experimental Setup and Results}\label{sec:expsetup}
Our experiments were performed on the Scikit-learn ``toy'' datasets MNIST and Diabetes \cite{sklearn}, with labels removed. We used the expression inside the minimization in (\ref{eqn:l2obj}) to evaluate performance. Our experiments considered estimating histograms in $d=2,3,4,5$ dimensional space. We consider two forms of dimensionality reduction. First we consider projecting the dataset onto its top $d$ principle components. As an alternative we consider projecting our dataset onto a random subspace of dimension $d$. We have constructed our random subspace dimensionality reduction so that each additional dimension adds a new index without affecting the others. For each dataset we randomly select an orthonormal basis that remains unchanged for all experiments $v_1,v_2,\ldots$. To transform a point $X$ to dimension $d$ we perform the following transform
\begin{align*}
  X_{\text{reduced dim.}  }=\left[
    v_1\cdots v_d
\right]^T X.
\end{align*}
 We consider both transforms since PCA may select dimensions where the features tend to be independent, as is the case when the distribution is a multivariate Gaussian. After dimensionality reduction we scale and translate the data to fit in the unit cube.

With our preprocessed dataset, each experiment consisted of randomly selecting 200 samples for training and using the rest to evaluate performance. For the estimators we tested all combinations using 1 to $\bmax$ bins per dimension and $k$ from 1 to $\kmax$. As $d$ increased the best cross validated $b$ and $k$ value decreased, so we reduced $\bmax$ and $\kmax$ for larger $d$ to reduce computational time, while still leaving a sizable gap between the best cross validated $b$ and $k$ across all runs of all experiment. For $d=2,3$ we have $\bmax = 15$ and $\kmax = 10$; for $d=4$ we have $\bmax = 12$ and $\kmax = 8$; for $d=5$ we have $\bmax = 8$ and $\kmax = 6$.
For parameter fitting we used random subset cross validation repeated 80 times using 40 of the 200 samples to cross validate. Performing 80 folds of cross validation was necessary because of the variance of the histogram estimated loss. This high variance is likely due to the noncontinuous nature of the histogram estimator itself and the noncontinuity of the histogram as a function of the data, i.e. slightly moving one training sample can potentially the change histogram bin in which it lies. Each experiment was run 32 times and we report the mean and standard deviations of estimator performance as well as the best parameters found from cross validation. We additionally apply the Wilcoxon signed rank test to the 32 pairs of performance results to statistically determine if the mean performance between the standard histogram and our algorithm are different and report the corresponding $p$-value. Our results are in Table \ref{tab:results} where the Tucker histogram dominates. We also observe that the Tucker histogram can estimate more bins per dimension than the standard histogram and is able to estimate more bins per dimension when the number of components of components is lower. This corroborates the number of components versus the bins per dimension trade-off from Section \ref{ssec:discussion}.
\section{Conclusion}
Through analysis of the histogram estimator, we have theoretically and empirically demonstrated that NNTF models can also be used to improve nonparametric density estimation. Though the histogram estimator is not a particularly popular estimator,we hope that the ideas presented here can be adapted to improve other techniques in nonparametric statistics.

  \bibliographystyle{abbrvnat}
\section*{Acknowledgements} 
This work was supported by the Berlin Institute for the Foundations of Learning and Data (BIFOLD) sponsored by the German Federal Ministry of Education and Research (BMBF).

\bibliography{main}
\appendix
  \section{Proofs Omitted from Main Text} \label{appx:proofs}
  \textbf{All norms are either the $\ell^1$, $L^1$, or total variation norm}, which are equivalent with respect to our analysis and the proper norm will be clear from context.

  \begin{proof}[Proof of Lemma \ref{lem:histcover}]
	  In Section 7.4 from \cite{devroye01}, the authors show that for any collection of measures $\mu_1,\ldots,\mu_b$, for all $\varepsilon >0$, that 
	  \begin{align*}
		  N\left( \conv\left(\left\{ \mu_1,\ldots, \mu_b \right\}\right), \varepsilon  \right) \le \left(b + \frac{b}{\varepsilon} \right)^b.
	  \end{align*}
	  With the additional assumption that $\varepsilon \le 1$ we have that $b + \frac{b}{\varepsilon} \le\frac{b}{\varepsilon}+\frac{b}{\varepsilon} =  \frac{2b}{\varepsilon}$ and thus
	  \begin{align*}
		  N\left( \conv\left(\left\{ \mu_1,\ldots, \mu_b \right\}\right)\right)\le \left(\frac{2b}{\varepsilon} \right)^b.
	  \end{align*}
	  If we let $\mu_i = \be_i$, the indicator vector at index $i$, then the lemma follows.
  \end{proof}

  \begin{proof}[Proof of Lemma \ref{lem:simplecover}]
	  From Lemma \ref{lem:histcover} we know there exists a finite collection of probability vectors $\tsP$ such that $\tsP$ is an $\varepsilon/d$-covering of $\Delta_b$ and $\left| \tsP\right| \le \left(\frac{2bd}{\varepsilon} \right)^b$. Note that the set $\left\{\tp_1\otimes\dots\otimes \tp_d: \tp_i \in \tsP \right\}$ contains at most $\left( \left(\frac{2bd}{\varepsilon} \right)^b\right)^d = \left(\frac{2bd}{\varepsilon} \right)^{bd}$ elements. We will now show that this set is an $\varepsilon$-cover of $\sT_{d,b}^1$. Let $p_1\ot \cdots \ot p_d \in \sT_{d,b}^1$ be arbitrary. From our construction of $\tsP$ there exist elements $\tp_1,\ldots,\tp_d \in \tsP$ such that $\left\|p_i - \tp_i\right\| \le \frac{\varepsilon}{d}$.

	  We will now make use of Lemma 3.3.7 in \cite{reiss89}, which states that, for any collection of probability vectors $q_1,\ldots, q_d$ and $\tq_1,\ldots, \tq_d$, the following holds
	  \begin{align*}
		  \left\|\prod_{i=1}^d q_i -\prod_{j=1}^d \tq_j\right\| \le \sum_{i=1}^d \left\|q_i- \tq_i \right\|.
	  \end{align*}
	  From this it follows that
	  \begin{align*}
		  \left\| \prod_{i=1}^d p_i - \prod_{j=1}^d \tp_j \right\|
		  \le \sum_{i=1}^d \left\|p_i - \tp_i\right\|
		  \le d \frac{\varepsilon}{d}
		  = \varepsilon
	  \end{align*}
	  thus completing our proof.
  \end{proof}

  \begin{proof}[Proof of Lemma \ref{lem:mixcover}]
	  Let $\tsP$ be the finite collection of probability measures with $|\tsP| = N\left(\sP,\varepsilon\right)$ which $\varepsilon$-covers $\sP$. Similarly let $W \subset \Delta_k$ with $|W| = N\left(\Delta_k,\delta\right)$  such that $W$ is a $\delta$-cover of $\Delta_k$. Consider the set 
	  \begin{align*}
		  \Omega = \left\{ \sum_{i=1}^k \tw_i \tp_i \middle| \tw \in W, \tp_i \in \tsP   \right\}.
	  \end{align*}
	  Note that this set contains at most $N\left(\sP,\varepsilon\right)^k N\left(\Delta_k,\delta\right)$ elements. We will now show that it $\left(\delta + \varepsilon\right)$-covers $k\mix\left(\sP\right)$, which completes the proof. Let $\sum_{i=1}^k p_i w_i \in k\operatorname{-mix}\left(\sP\right)$. We know there exists elements $\tp_1,\ldots,\tp_k \in \tsP$ such that $\left\|\tp_i - p_i\right\|\le \varepsilon$ and $\tw \in W$ such that $\left\| w - \tw \right\| \le \delta$ and thus $\sum_{i=1}^k \tp_i \tw_i \in \Omega$. Now observe that
	  \begin{align*}
		  \left\|\sum_{i=1}^k \tp_i \tw_i - \sum_{j=1}^{k} p_j  w_j\right\|
		  &= \left\|\sum_{i=1}^k \tp_i \tw_i - \sum_{j=1}^k p_j \tw_j + \sum_{l=1}^k p_l \tw_l - \sum_{r=1}^{k} p_r  w_r\right\| \\
		  &\le \left\|\sum_{i=1}^k \left(\tp_i - p_i\right) \tw_i     \right\|+ \left\| \sum_{i=1}^k p_i \left(\tw_i -  w_i\right)\right\| \\
		  &\le \sum_{i=1}^k \tw_i \left\| \tp_i - p_i \right\|+  \sum_{i=1}^k  \left|\tw_i -  w_i\right| \\
		  &\le \sum_{i=1}^k \tw_i \varepsilon + \left\|\tw -  w\right\| \\
		  &\le \varepsilon + \delta.
	  \end{align*}
  \end{proof}

  \begin{proof}[Proof of Lemma \ref{lem:lrcover}]
	  Note that $\sT_{d,b}^k = k\operatorname{-mix}\left(\sT_{d,b}^1\right)$.
	  Applying Lemma \ref{lem:mixcover} followed by Lemmas \ref{lem:histcover} and \ref{lem:simplecover} we have that
	  \begin{align*}
		  N\left(\sT_{d,b}^k, \varepsilon \right) 
		  \le N\left(\sT_{d,b}^1, \varepsilon/2 \right)^k N\left(\Delta_k, \varepsilon/2 \right)
		  \le \left(\frac{4bd}{\varepsilon} \right)^{bdk} \left( \frac{4k}{\varepsilon}\right)^k.
	  \end{align*}
  \end{proof}

  \begin{proof}[Proof of Lemma \ref{lem:tuckcover}]
	  Fix $k,d,b$ and $0<\varepsilon \le 1$. We are going to construct an $\varepsilon$-cover of $\tsT_{d,b}^k$.
	  From Lemma \ref{lem:histcover} we know that there exists a set $\sB \subset \Delta_b$ which $\left(\frac{\varepsilon}{2d}\right)$-covers of $\Delta_b$ and contains no more than $\left(\frac{4bd}{\varepsilon}\right)^b$ elements. Let $\sP$ be the collection of all $d\times k$ arrays whose entries are elements from $\sB$. So we have that 
	  \begin{align*}
		  \left|\sP\right| = \left|\sB\right|^{dk}\le \left(\frac{4bd}{\varepsilon}\right)^{bdk}.
	  \end{align*}

	  From Lemma \ref{lem:histcover} there exists $\sW$ which is an $\varepsilon/2$-cover of $\sT_{d,k}$ and contains no more than $\left(4k^d/\varepsilon\right)^{\left(k^d\right)}$ elements. Now let
	  \begin{align*}
		  \sL_{d,b}^k =  \left\{ \sum_{S \in [k]^d} \tW_S \prod_{i=1}^d \tp_{i,S_i} \middle| \tW \in \sW, \tp\in \sP \right\}.
	  \end{align*}
	  Note that 
	  \begin{align*}
		  \left| \sL_{d,b}^k\right| \le \left|\sW\right| \left|\sP\right| \le \left(\frac{4k^d}{\varepsilon} \right)^{k^d} \left(\frac{4bd}{\varepsilon}\right)^{bdk}.
	  \end{align*}
	  We will now show that $\sL_{d,b}^k$ is an $\varepsilon$-cover of $\tsT_{d,b}^k$. To this end let $\sum_{S\in \left[k\right]^d} W_S \prod_{i=1}^{d}p_{i,S_i} \in \tsT_{d,b}^k$ be arbitrary, where $W\in \sT_{d,k}$ and $p_{i,j} \in \Delta_b$. From our construction of $\sW$, there exists $\tW \in \sW$ such that $\left\|W - \tW\right\| \le \varepsilon/2$. There also exists $\tp \in \sP$ such that $\left\|\tp_{i,j} - p_{i,j}\right\| \le \varepsilon/2$ for all $i,j$. Therefore we have that
	  \begin{align*}
		  \sum_{S \in \left[k\right]^d} \tW_S \prod_{i=1}^d \tp_{i,S_i} \in \sL_{d,b}^k.
	  \end{align*}
	  So finally 
	  \begin{align*}
		  &\left\|\sum_{S\in \left[k\right]^d} W_S \prod_{i=1}^{d}p_{i,S_i} - \sum_{R \in \left[k\right]^d} \tW_R \prod_{j=1}^d \tp_{j,R_j} \right\|\\
		  &\le \left\|\sum_{S\in \left[k\right]^d} W_S \prod_{i=1}^{d}p_{i,S_i} - \sum_{R \in \left[k\right]^d} W_R \prod_{j=1}^d \tp_{j,R_j} \right\|
		  +\left\|\sum_{S \in \left[k\right]^d} W_S \prod_{i=1}^d \tp_{i,S_i} - \sum_{R \in \left[k\right]^d} \tW_R \prod_{j=1}^d \tp_{j,R_j} \right\|\\
		  &\le\sum_{S\in \left[k\right]^d}W_S \left\|  \prod_{i=1}^{d}p_{i,S_i} - \prod_{j=1}^d \tp_{j,S_j} \right\|
		  +\sum_{R \in \left[k\right]^d} |W_R- \tW_R|\left\|   \prod_{j=1}^d \tp_{j,R_j} \right\|\\
		  &\le\sum_{S\in \left[k\right]^d}W_S \varepsilon
		  +\left\|W - \tW\right\|\\
		  &\le \varepsilon/2+ \varepsilon/2 = \varepsilon.
	  \end{align*}

  \end{proof}

  \begin{proof}[Proof of Theorem \ref{thm:genrate}]
	  We will be applying the estimator from Lemma \ref{lem:finiteest} to a series of $\delta$-covers of $\sH_{d,b}^k$. We begin by constructing a series of $\delta$-covers whose cardinality doesn't grow too quickly. Corollary \ref{cor:lrcover} states that, for all $0< \delta \le 1$, that $N\left(\sH_{d,b}^k,\delta\right) \le \left( \frac{4bd}{\delta} \right)^{bdk} \left( \frac{4k}{\delta}\right)^k$. \emph{For sufficiently large} $b$ and $k$ and \emph{sufficiently small} $\delta$, the following holds
	  \begin{align}
		  \log \left(\left(\frac{4bd}{\delta} \right)^{bdk} \left( \frac{4k}{\delta}\right)^k \right)\notag
		  & =bdk \log\left(\frac{4bd}{\delta}\right) + k \log\left(\frac{4k}{\delta} \right)\notag\\
		  & = bdk\left[\log \left(b\right) + \log\left(\frac{4d}{\delta}\right)\right] + k \left[\log\left(k\right) + \log\left(\frac{4}{\delta} \right)\right]\notag\\
		  & \le bdk\left[\log \left(b\right) + \log\left(b\right)\log\left(\frac{4d}{\delta}\right)\right] + dk \left[\log\left(k\right) + \log \left(k\right)\log\left(\frac{4d}{\delta} \right)\right]\notag\\
		  & = \left(bk \log\left( b\right) + k\log\left(k\right)\right) d\left( 1+\log \left(\frac{4d}{\delta}\right) \right) .\label{eqn:same}
	  \end{align}
	  Using the argument from the proof of Lemma \ref{lem:finiteest} we have that, because  $n/(bk\log(b) + k\log(k)) \to \infty$ there exists a sequence of positive values $C = C(n)$ such that $C \to \infty$ and $n>C\left[bk\log(b) + k\log(k)\right]$. If we let $\delta = \frac{4d}{\exp\left(\frac{C}{d}-1\right)}$ we have that $\delta \to 0$ and
	  \begin{align*}
		  \left(bk \log\left( b\right) + k\log\left(k\right)\right) d\left( 1+\log \left(\frac{4d}{\delta}\right) \right) \le n.
	  \end{align*}
	  Because of this we can construct collections of densities $\tsP_n\subset \sH_{d,b}^k$ such that $\tsP_n$ is a $\delta$-covering of $\sH_{d,b}^k$ with $\left| \tsP\right| \to \infty$, $n/\log\left|\tsP_n\right| \to \infty$ and $\delta \to 0$. Let $V_n$ be the estimator from Lemma \ref{lem:finiteest} applied to the sequence $\tsP_n$.

	  Let $\varepsilon >0$ be arbitrary. Due to the way that we have constructed the sequence $\tsP_n$, for sufficiently large $n$, we have that $3\sup_{q \in \sH_{d,b}^k} \min_{\tq \in \tsP_n}\left\|q - \tq \right\| \le \varepsilon/2$. It therefore follows that, for sufficiently large $n$, the following holds for all $p \in \sD_d$
	  \begin{align*}
		  3 \min_{q \in \sH_{d,b}^k} \left\|p-q\right\| + \varepsilon
		  &\ge 3 \min_{q \in \sH_{d,b}^k} \left\|p-q\right\| + 3\sup_{q \in \sH_b^k} \min_{\tq \in \tsP_n}\left\|q - \tq \right\|+ \varepsilon/2 \\
		  &\ge 3 \min_{q \in \sH_{d,b}^k}\left[ \left\|p-q\right\| + \min_{\tq \in \tsP_n}\left\|q - \tq \right\|\right]+ \varepsilon/2 \\
		  &= 3 \min_{q \in \sH_{d,b}^k}\min_{\tq \in \tsP_n} \left\|p-q\right\| + \left\|q - \tq \right\|+ \varepsilon/2 \\
		  &\ge 3 \min_{\tq \in \tsP_n} \left\|p- \tq \right\|+ \varepsilon/2.
	  \end{align*}
	  From this we have that, for sufficiently large $n$
	  \begin{align*}
		  \sup_{p \in \sD_d }P\left(\left\|V_i - p\right\| > 3 \min_{q\in \sH_{d,b}^k}\left\|p -q \right\| + \varepsilon  \right)
		  \le \sup_{p \in \sD_d }P\left(\left\|V_i - p\right\| > 3 \min_{\tq \in \tsP_n} \left\|p- \tq \right\|+ \varepsilon/2  \right)
	  \end{align*}
	  and the right side goes to zero due to Lemma \ref{lem:finiteest}, thus completing the proof.
  \end{proof}

  \begin{proof}[Proof of Theorem \ref{thm:tuckrate}]
	  This proof is very similar to the proof of Theorem \ref{thm:genrate}. We will be applying the estimator from Lemma \ref{lem:finiteest} to a series of $\delta$-covers of $\tsH_{d,b}^k$. We begin by constructing a series of $\delta$-covers whose cardinality doesn't grow too quickly. Corollary \ref{cor:tuckcover} states that, for all $0< \delta \le 1$, that $N\left(\tsH_{d,b}^k,\delta\right) \le\left(\frac{4bd}{\delta} \right)^{bdk} \left( \frac{4k^d}{\delta}\right)^{k^d}$. \emph{For sufficiently large} $b$ and $k$ and \emph{sufficiently small} $\delta$, the following holds
	  \begin{align*}
		  \log \left(\left(\frac{4bd}{\delta} \right)^{bdk} \left(\frac{4k^d}{\delta} \right)^{k^d} \right)
		  &=bdk \log \left(\frac{4bd}{\delta}\right) + k^d \log \left(\frac{4k^d}{\delta} \right)\\
		  &\le d\left(bk \log \left(\frac{4bd}{\delta}\right) + k^d \log \left(\frac{4k^d}{\delta}\right) \right)\\
		  &= d\left(bk \left(\log(b) + \log \left(\frac{4d}{\delta}\right)\right) + k^d \left( \log\left(k^d\right) + \log\left(\frac{4}{\delta}\right)  \right) \right)\\
		  &\le d\left(bk \left(\log(b)+ \log \left(\frac{4d}{\delta}\right)\right) + k^d \left( \log\left(k^d\right) + \log\left(\frac{4d}{\delta}\right)  \right) \right)\\
		  &= \left(bk \log(b) + k^d  \log\left(k^d\right)  \right)d\left(1 + \log \left(\frac{4d}{\delta}\right)\right).
	  \end{align*}
	  Note that replacing $bk \log\left(b\right) + k \log\left(k\right)$ with $bk \log \left(b\right)  + k^d  \log\left(k^d\right)$ in the last line is exactly (\ref{eqn:same}) in our proof of Theorem \ref{thm:genrate} . From here we can proceed exactly as in the proof of Theorem \ref{thm:genrate} by replacing $\sH_{d,b}^k$ with $\tsH_{d,b}^k$ and $bk \log\left(b\right) + k \log\left(k\right)$ with $bk \log \left(b\right)  + k^d  \log\left(k^d\right)$.
  \end{proof}

  \begin{proof}[Proof of Lemma \ref{lem:lrbias}]
	  Let $\varepsilon >0$. Theorem 5 in Chapter 2 of \cite{gyorfi85} states that, for any $p \in \sD_d$, that $\min_{h\in \sH_{d,b}}\left\|p-h\right\| \to 0$ as $b\to \infty$, i.e. the bias of a histogram estimator goes to zero as the number of bins per dimension goes to infinity. Thus there exists a sufficiently large $B$ such that there exists a histogram $h \in \sH_{d,B}$ which is a good approximation of $p$, $\left\|p - h \right\| <\varepsilon/2$. In this proof we we will argue that once $k\ge B^d$ and $b$ is sufficiently large, we can find an element of $\sH_{d,k}^k$ where the multi-view components can approximate the $k$ bins of $h$.

	  We have that 
	  \begin{align*}
		  h = \sum_{A \in \left[B\right]^{\times d}} w_A \prod_{i=1}^d h_{d,B,A_i}.
	  \end{align*}
	  From the same theorem there exists $a_0$ such that, for all $a\ge a_0$,  for all $i$, there exists  $\th_{1,a,i} \in \sH_{1,a}$ such that $\left\|h_{1,B,i} -\th_{1,a,i} \right\| < \varepsilon/(2d)$ for all $i\in [B]$. For any multi-index $A \in [B]^d$, we define
	  \begin{align*}
		  \th_{d,a,A} = \prod_{j=1}^d \th_{1,a,A_j}.
	  \end{align*}
	  Now we have that, for all $a \ge a_0$ and $A \in [B]^d$,
	  \begin{align}
		  \left\|h_{d,B,A} - \th_{d,a,A}\right\|
		  & = \left\| \prod_{i=1}^d h_{1,B,A_i} - \prod_{j=1}^d \th_{1,a,A_j}\right\|\notag \\
		  & \le \sum_{i=1}^d \left\|h_{1,B,A_i} - \th_{1,a,A_i} \right\|\label{eqn:supp}\\
		  & \le d \frac{\varepsilon}{2d} \notag\\
		  & = \varepsilon/2,\notag
	  \end{align}
	  where we use the previously mentioned product measure inequality for (\ref{eqn:supp}).
	  As soon as $k\ge B^d$ and $a\ge a_0$ the set $\sH_{d,a}^k$ contains the element,
	  \begin{align*}
		  \th \triangleq \sum_{A \in \left[B\right]^{\times d}} w_A \th_{d,a,A}.
	  \end{align*}
	  Now we have that, for all $a\ge a_0$.
	  \begin{align*}
		  \left\|h - \th\right\|
		  & = \left\|\sum_{A \in \left[B\right]^{\times d}} w_A h_{d,B,A} - \sum_{Q \in \left[B\right]^{\times d}} w_Q \th_{d,a,Q}  \right\|\\
		  & \le \sum_{A \in \left[B\right]^{\times d}}w_A \left\|  h_{d,B,A} -  \th_{d,a,A}  \right\|\\
		  & \le \varepsilon/2.
	  \end{align*}
	  From the triangle inequality we have that 
	  \begin{align*}
		  \left\|p - \th \right\| \le 
		  \left\|p- h \right\| + \left\| h- \th\right\| \le\varepsilon.
	  \end{align*}
	  So we have that, for sufficiently large $b$ and $k$ 
	  \begin{align*}
		  \min_{q \in \sH_{d,b}^k} \left\| p - q \right\| \le \varepsilon
	  \end{align*}
	  which completes our proof.
  \end{proof}

  \begin{proof}[Proof of Lemma \ref{lem:tuckbias}]
	  We will show that $\sH_{d,b}^k \subset \tsH_{d,b}^k$ and the theorem clearly follows due to Lemma \ref{lem:lrbias}. Any element of $\sH_{d,b}^k$ will have the following representation
	  \begin{align}
		  \sum_{i=1}^k w_i \prod_{j=1}^d f_{i,j}: w \in\Delta_k, f_{i,j} \in \sH_{1,b}.\label{eqn:tuckbias}
	  \end{align}
	  Letting $W \in \sT_{d,k}$ with $W_{i,\ldots,i} = w_i$ for all $i$ and the rest of the entries of $W$ be zero and letting $\tf_{j,i} = f_{i,j}$ for all $i,j$ we have that 
	  \begin{align*}
		  \sum_{S\in [k]^d} W_S \prod_{j=1}^d \tf_{j,S_j}
		  &= \sum_{i=1}^k W_{i,\ldots,i}\prod_{j=1}^d \tf_{j,i}\\
		  &= \sum_{i=1}^k w_i \prod_{j=1}^d f_{i,j}
	  \end{align*}
	  so we have that (\ref{eqn:tuckbias}) is an element of $\tsH_{d,b}^k$ and we are done.
  \end{proof}

  \begin{proof}[Proof of Theorem \ref{thm:lower}]
	  We will proceed by contradiction. Suppose $V_n$ is an estimator violating the theorem statement, i.e. there exist sequences $b\to \infty$ and $k\to \infty$ with $n/\left(bk\right) \to 0$ and $b\ge k$ such that, for all $\varepsilon >0$,
	  \begin{align*} 
		  \sup_{p \in \sD_d }P\left(\left\|V_n - p\right\| > 3 \min_{q\in \sH_{d,b}^k}\left\|p -q \right\| + \varepsilon  \right) \to 0.
	  \end{align*}

	  Let $\left(p_n\right)_{n=1}^\infty$ be a sequence of probability vectors $p_n \in \Delta_{b(n)\times k(n)}$ which represent distributions over $\left[b(n)\right] \times \left[k(n)\right]$. Let $\sX_n \triangleq  \left(X_{n,1},\ldots,X_{n,n} \right)$ with $X_{n,1},\ldots,X_{n,n} \simiid p_n$.

	  We will now construct a series of estimators for $p_n$ using $V_n$. Let $\tsX_n = \left(\tX_{n,1},\ldots, \tX_{n,n} \right)$ which are independent random variables with $\tX_{n,i} \sim h_{d,b,\left(X_{n,i},1,\ldots,1\right)}$. For this proof we will assume $d > 2$ but the proof can be simplified in a straightforward manner to the $d=2$ case by ignoring the indices and modes beyond the second. Note that that $X_{n,i}$ contains two indices.  Now we have the following for the densities of $\tX_{n,i}$
	  \begin{align}
		  p_{\tX_{n,i}} \notag
		  &= \sum_{(j,\ell) \in \left[b\right] \times \left[k\right]} p_{\tX_{n,i} | X_{n,i} = \left(j,\ell\right)} P(X_{n,i} = \left(j,\ell\right))\\\notag
		  &= \sum_{(j,\ell) \in \left[b\right] \times \left[k\right]} h_{d,b,\left(j,\ell,1,\ldots, 1\right)}  p_{n}\left(j,\ell\right)\\\notag
		  &= \sum_{ \ell \in \left[k\right]} \sum_{ j \in \left[b\right]} h_{d,b,\left(j,\ell,1,\ldots, 1\right)}  p_{n}\left(j,\ell\right)\\\notag
		  &= \sum_{ \ell \in \left[k\right]} \sum_{ j \in \left[b\right]}p_{n}\left(j,\ell\right) h_{1,b,j} \otimes h_{1,b,\ell} \otimes \prod_{a \in [d-2]}h_{1,b,1}\notag  \\
		  &= \sum_{ \ell \in \left[k\right]} \left(\sum_{ j \in \left[b\right]}p_{n}\left(j,\ell\right) h_{1,b,j}\right) \otimes h_{1,b,\ell} \otimes \prod_{a \in [d-2]}h_{1,b,1} \label{eqn:histform} \\
		  &= \sum_{ \ell \in \left[k\right]}\left(\sum_{ q \in \left[b\right]}p_{n}\left(q,\ell\right)\right) \left(\sum_{ j \in \left[b\right]}\frac{p_{n}\left(j,\ell\right)}{\sum_{ q \in \left[b\right]}p_{n}\left(q,\ell\right)} h_{1,b,j}\right) \otimes h_{1,b,\ell} \otimes \prod_{a \in [d-2]}h_{1,b,1}. \label{eqn:khistform}
	  \end{align}
	  This last line is in the form of (\ref{eqn:lrhist}) in the main text and is thus an element of $\sH_{d,b}^k$. To see this we will show the correspondence between the terms in (\ref{eqn:khistform}) from here and the terms in (\ref{eqn:lrhist}) in the main text:
	  \begin{align*}
		  w_\ell &:= \left(\sum_{ q \in \left[b\right]}p_{n}\left(q,\ell\right)\right)\\
		  f_{\ell,1} &:= \left(\sum_{ j \in \left[b\right]}\frac{p_{n}\left(j,\ell\right)}{\sum_{ q \in \left[b\right]}p_{n}\left(q,\ell\right)} h_{1,b,j}\right) \\
		  f_{\ell,2} &:= h_{1,b,\ell}\\
		  f_{i,j} &:= h_{1,b,1}, \forall j>2, \forall i.
	  \end{align*}
	  Let $V_n$ estimate $\tP_n \triangleq p_{\tX_{n,i}}$ so $\tX_{n,1}, \ldots , \tX_{n,n} \simiid \tP_n$. We will use $V_n$ to construct an estimator $v_n$ for $p_n$.

	  Because $\tP_n \in \sH_{d,b}^k$ \footnote{We will use this portion of the proof again for our proof of Theorem \ref{thm:tucklower} \label{fn:lower}} for all $n$ we have that $\left\|V_n - \tP_n\right\| \cip 0$ and thus $\left\| U_{d,b}^{-1}(V_n) - U_{d,b}^{-1}(\tP_n)\right\|\cip 0 $. Note that $\left[U_{d,b}^{-1}(\tP_n)\right]_{j,\ell,A} = p_n(j,\ell)$ when $A=\left(1,\ldots,1\right)$ and zero otherwise (see (\ref{eqn:histform})). We define the linear operator $B_n: \sT_{d,b} \to \Delta_{b\times k}$ as 

	  \begin{align*}
		  \left[B_n(T)\right]_{j,\ell} \triangleq \sum_{A \in \left[b\right]^{\times d - 2}} T_{j,\ell, A}
	  \end{align*}
	  i.e. the linear operator which sums out all modes except for the first two.
	  We have that $B_n(U_{d,b}^{-1}(\tP_n)) = p_n$. Now let $v_n = B_n(U_{d,b}^{-1}(V_n))$ be the estimator for $p_n$. Now we have that
	  \begin{equation*} \label{eqn:lower}
		  \left\|v_n - p_n\right\|
		  = \left\|B_n(U_{d,b}^{-1}(\tP_n)) - B_n(U_{d,b}^{-1}(V_n)) \right\|= \left\|B_n(U_{d,b}^{-1}(\tP_n - V_n)) \right\| .
	  \end{equation*}
	  We have that $B_n$ is a nonexpansive operator due to the triangle inequality,
	  \begin{align*}
		  \left\|B_n\left(T\right)\right\|
		  = \sum_{j,l} \left|\sum_{A \in \left[b\right]^{\times d - 2}} T_{j,\ell, A}  \right| \le \sum_{j,l}\sum_{A \in \left[b\right]^{\times d - 2}} \left| T_{j,\ell, A}  \right|
		  = \left\|T\right\|,
	  \end{align*}
	  so the operator norm of $B_n$ is less than or equal to one. We also know that $U^{-1}_{d,b}$ an isometry and $\left\|\tP_n - V_n\right\|\cip 0$, so it follows that $\left\|v_n - p_n\right\|\cip 0$ for any sequence of $p_n\in \Delta_{\left[b(n)\right] \times \left[k(n)\right]}$. We will now use following theorem from \cite{han15} to show that no such estimator $v_n$ can exist.

	  \begin{thm}[\cite{han15} Theorem 2.]
		  For any $\zeta \in \left(0,1\right]$, we have
		  \begin{align*}
			  &\inf_{\hat{p}} \sup_{p\in \Delta_a} \E_p\left\|\hat{p}-p\right\| \ge 
			  \frac{1}{8} \sqrt{\frac{ea}{\left(1+\zeta\right)n}} \1 \left( \frac{\left(1 + \zeta\right)n}{a}>\frac{e}{16}\right)\\
			  &\quad+ \exp\left(- \frac{2\left(1 + \zeta\right)n}{a}\right) \1 \left( \frac{\left(1 + \zeta\right)n}{a}\le\frac{e}{16}\right)
			  - \exp\left(-\frac{\zeta^2n}{24}\right) - 12 \exp\left(- \frac{\zeta^2 a}{32 \left(\log a \right)^2}\right)
		  \end{align*}
		  where the infimum is over all estimators.
	  \end{thm}
	  Our estimator is equivalent to estimating a categorical distribution with $a = bk$ categories. Letting $\zeta = 1$, $bk \to \infty$, and $n\to \infty$, with $n/\left(bk\right) \to 0$, we get that for sufficiently large $n$
	  \begin{align*}
		  \inf_{\hat{p}} \sup_{p\in \Delta_{bk}} \E_p\left\|\hat{p}-p\right\| \ge \exp\left(- \frac{4n}{bk}\right) - \exp\left(-\frac{n}{24}\right) - 12 \exp\left(- \frac{bk}{32 \left(\log bk \right)^2}\right)
	  \end{align*}
	  whose right hand side converges to 1. From this we get that 
	\begin{align*}
		\liminf_{n\to \infty} \sup_{p_n\in \Delta_{bk}} \E_{p_n}\left\|v_n-p_n\right\|  > \frac{1}{2}
	\end{align*}
	which contradicts $\left\|v_n - p_n\right\|\cip 0 $ for arbitrary sequences $p_n$.
\end{proof}

  \begin{proof}[Proof of Thoerem \ref{thm:tucklower}]

	  We will proceed by contradiction. Suppose $V_n$ is an estimator violating the theorem statement, i.e. there exist sequences $b\to \infty$ and $k\to \infty$ with $n/\left(bk + k ^d\right) \to 0$ and $b\ge k$ such that, for all $\varepsilon >0$,
	  \begin{align*}
		  \sup_{p \in \sD_d }P\left(\left\|V_n - p\right\| > 3 \min_{q\in \tsH_{d,b}^k}\left\|p -q \right\| + \varepsilon  \right) \to 0.
	  \end{align*}
	  Since $n/(bk + k^d) \to 0$ we have that $(bk + k^d)/n \to \infty$ so there is a subsequence $n_i$ such that $b(n_i)k(n_i)/n_i \to \infty$ or $k(n_i)^d/n_i\to \infty$, or equivalently $n_i/(b(n_i)k(n_i)) \to 0$ or $n_i/k(n_i)^d \to 0$. We will show that both cases lead to a contradiction. We will let $b$ and $k$ be functions of $n_i$ now implictly when defining limits. \\
	  \textbf{Case $n_i/(bk) \to 0$:}
	  We proceed similarly to the proof of Theorem \ref{thm:lower}. Let $\left(p_n\right)_{n=1}^\infty$, $\tP_n$, and $\sX_n$ be defined as in the proof of Theorem \ref{thm:lower}. Note that $\sH_{d,b}^k \subset \tsH_{d,b}^k$ (see proof of Lemma \ref{lem:tuckbias}) and thus $\tP_n \in \tsH_{d,b}^k$. We can proceed exactly as in our proof of Theorem \ref{thm:lower} at footnote \ref{fn:lower}, by simply replacing $\sH_{d,b}^k$ with $\tsH_{d,b}^k$ and $n$ with $n_i$ which finishes this case.\\
	  \textbf{Case $n_i/k^d \to 0$:} Let $(p_n)_{n=1}^\infty$ be a sequence of elements in $\sT_{d,k}$ which represents distributions over $[k]^d$. Let $\sX_n \triangleq  \left(X_{n,1},\ldots,X_{n,n} \right)$ with $X_{n,1},\ldots, X_{n,n}\simiid p_n$. Let $\tsX_n = \left(\tX_{n,1},\ldots, \tX_{n,n} \right)$ which are independent random variables with $\tX_{n,i}\sim h_{d,b,X_{n,i}}$. Let $\tP_n$ be the density for $\tX_{n,i}$. Note that $k\le b$. So we have that
	  \begin{align*}
		  \tP_n 
		  &= \sum_{S\in [k]^d} p_{\tX_{n,i}|X_{n,i}= S}P(X_{n,i} = S)\\
		  &= \sum_{S\in [k]^d} h_{d,b,S}p_n(S)\\
		  &= \sum_{S\in [k]^d}p_n(S)\prod_{i=1}^d h_{1,b,S_i}
	  \end{align*}
	  and thus $\tP_n \in \tsH_{d,b}^k$. We proceed as in Theorem \ref{thm:lower} to find an estimator for elements of $\sT_{d,k}$ which is equivalent to estimating elements of $\Delta^{k^d}$ which is impossible since $n_i/k^d \to 0$.
  \end{proof}

\end{document}